\documentclass[10pt]{amsart}

\usepackage{amsmath, amssymb, amsthm, mathrsfs, float
}

\newtheorem{thm}{Theorem}

\newtheorem{cor}[thm]{Corollary}

\newtheorem{prop}[thm]{Proposition}
\newtheorem{rem}[thm]{Remark}

\newcommand{\N}{\ensuremath{\mathbb{N}}}
\newcommand{\A}{\ensuremath{\mathrm{A}}}
\newcommand{\Oh}{\ensuremath{\mathrm{O}}}
\newcommand{\oh}{\ensuremath{\mathrm{o}}}
\newcommand{\nP}{\ensuremath{\mathcal{P}}}

\newtheorem*{lemA}{Lemma A}
\newtheorem*{lemB}{Lemma B}

\renewcommand{\l}{\ell}



\begin{document}

\title{On the distribution of totients $2$ mod. $4$}

\author{Andr\'{e} Contiero}
\address{Departamento de Matem\'atica, ICEx, UFMG Av. Ant\^onio Carlos 6627, 30123-970 Belo Horizonte MG, Brazil}
\email{contiero@ufmg.br\footnote{Partially supported by \emph{Universal CNPq 486468/2013-5 } and by \emph{Programa Institucional de Aux\'ilio \`a Pesquisa de Docentes Rec\'em-Contratados - UFMG}}
}
\author{Davi Lima}
\address{Instituto de Matem\'atica, UFAL. Av. Lourival de Melo Mota, s/n, 57072-900 Macei\'o AL, Brazil}
\email{davimat@impa.br}

\dedicatory{}

\begin{abstract}

In this paper we study the distribution of totients $2 \, \mbox{mod.} 4$.
We prove that the asymptotic magnitude of such totients with multiplicity two is half of that of prime numbers.
As a corollary we obtain that the relative asymptotic density of the number of those totients with multiplicity four 
over the number of totients with multiplicity two is zero. We also obtain that the set of totients with multiplicity $k>1$ which 
have power, bigger than one, of a prime in their pre-images has relative asymptotic density zero over the number of all totients of 
multiplicity $k$. A result on the distribution of consecutive pairs of totients $2\,\mbox{mod.} 4$, which relates to cousin primes, is also provided.


%
%
\end{abstract}

\maketitle

\section{Introduction}

One of the main functions in number theory is the widely known Euler's \textit{totient} $\phi$-function.
A particular subject of study is the set $\mathcal{V}$ of totients, ie. the set of the images taken by Euler's $\phi$-function, 
$$\mathcal{V}:=\{1,2,4,6,8,10,12, 16, 18,\dots\}\,.$$ The distribution
of totients has been investigated for many authors and from many perspectives, it is also
closely related to deep conjectures involving Euler's totient function, in particular, a famous
Carmichael's conjecture \cite{C} which states that there is no $m$ such that $\A(m)=1$. 
Here $\A(m):=\vert \phi^{-1}(m)\vert$ is the multiplicity of $m$.
Given an integer $k>1$, the well known Sierpi\'nski's conjecture, actually Theorem, says that there is a number $m$ such that $\A(m)=k$, 
this was proved by Ford in \cite{F}. 

In a very deep study on the distribution of totients Ford \cite{F1} provided the main results already known 
on distribution of totients, he also summarized some of the main previously known results.
Related to the proportion of totients with multiplicity $k>1$,
Ford proved that for every $\epsilon>0$, $\vert\mathcal{V}_{k}(x)\vert\gg_{\epsilon}\vert\mathcal{V}(x)\vert m^{-1-\epsilon}$,
where $\A(m)=k$, cf. \cite[Thm. 2]{F2}. For a positive integer number $k$,  the set $\mathcal{V}_{k}$ stands for 
totients whose multiplicity is exactly $k$, $$\mathcal{V}_{k}:=\{m\in\mathbb{N};\,\A(m)=k\}.$$ In the same paper, Ford encourages to 
classify totients more finely, as can be noted in last paragraph of page $39$ (\cite[pg. 39]{F2}).

Totients in suitable residue class have particular interest. It was proved by Dence and Pomerance 
\cite[Thm. 1.1]{PD} that if a residue class contains a multiple of 4, then it must contain infinitely many totients, they also got 
an asymptotic formulae for totients in a residue class modulo $12$, see \cite[Thm. 1.2]{PD}.

The set of totients bigger than $1$ are divided in two classes modulo $4$, namely $0$ and $2$.
In this paper we focus on the distribution of totients which are $2$ modulo $4$.
This class of totients has already been studied by Klee in \cite{K} from a different point of view, he did not consider its distribution.
We also note that Mingzhi characterized the nontotients $2$ modulo $4$, cf. \cite[Thm. 2]{Min}.
It is quite simple to prove that totients $2$ modulo $4$ have multiplicity equals to $0$, $2$ or $4$, see our Lemma B below. This led us to
introduce the sets $$\mathcal{T}_{k}:=\{2r\in\mathbb{N};\, r\mbox{ odd and }\ \A(2r)=k\}\ \ \ \ (k=0,2,4).$$
Having in hand a result on a relative asymptotic density of all totients whose pre-images possess a power (bigger than one) of an odd prime number, see Lemma A, we can
prove that $\mathcal{T}_2(x)$ has magnitude $\pi(x)/2$, where $\pi(x)$ is the number of primes numbers
not bigger than $x$, this is our Theorem \ref{thmA}. As an immediate corollary we can show that 
$\vert\mathcal{T}_{4}(x)\vert=\oh(\vert\mathcal{T}_{2}(x)\vert)$, see Corollary \ref{corA}. Taking
the set  $\mathcal{V}_{k}^{l}$ of totients with multiplicity $k$ such that there is a power of an odd prime in the inverse image by $\phi$,
we prove that the limit of $\vert\mathcal{V}_{k}^{l}(x)\vert$ over $\vert\mathcal{V}_{k}(x)\vert$ is equal to zero when $x$ goes to infinity, cf. Proposition \ref{vkl}.
Moreover, Proposition \ref{pairs} concerns a distribution of consecutive pairs of totients $2\,\mbox{mod.} 4$ and relates to cousin primes.

\section{Two Lemmas}

As usual, the set of prime numbers is denoted by $\nP$ and $\pi(x)$ stands for the number
of prime numbers not bigger than $x$. We also use the \textit{big} $\Oh$  and \textit{small} $\oh$
standard notations. The following notation will be useful throughout this paper: given any subset 
$\mathrm{U}$ of the positive integers and $x\in\mathbb{R}$ a real number, 
$\mathrm{U}(x)$ stands for the elements of $\mathrm{U}$
not bigger than $x$, $$\mathrm{U}(x):=\{n\in\mathrm{U};\,n\leq x\}\,,$$ and it is clear that $\vert U(x)\vert$ denotes
the number of elements of the set $U(x)$.

\begin{lemA}\label{lemmaA}
Given an integer number $t>0$, let us consider the set 
$$\mathcal{R}_{t}:=\{k\in\N\,;\,p^i\in \phi^{-1}(k)\mbox{ with } i\ge t+1,\ p\in\nP,\ p>2\}\subset\mathcal{V}\,.$$ We have
$$\vert\mathcal{R}_{t}(x)\vert=\oh(\pi(\sqrt[t]{x}))\,.$$
\end{lemA}

\begin{proof}
We start by noting that the function $f:[t,+\infty)\rightarrow \mathbb{R}$, given by $f(t)=q^{t+1}-q^{t}$ 
is increasing when $q>1$. By the very definition, for every $k\in \mathcal{R}_{t}(x)$ there is an odd prime number $q$
and an integer $m\geq t+1$ such that 
$$x \geq k=q^m-q^{m-1}\ge q^{t+1}-q^{t}.$$
From the above inequality we get a upper bound for $q$, namely $q\leq \sqrt[t]{x/2}$. 
Since $x\geq k\geq 2q^{t}$, we get the following upper bound $t\le \lceil \log x/\log 3\rceil$. 
Now, let us take the set $\mathcal{U}(x):=\{q^j\le x; q \ \mbox{is prime}\}$. Noting that the function $k\mapsto q^j\in \phi^{-1}(k)$ 
is injective (by choosing the largest prime $q$, for example), we can see that $\vert\mathcal{R}_{t}(x)\vert\le \vert\mathcal{U}(x)\vert$. Hence
$$|\mathcal{U}(x)|\leq \pi(\sqrt[t]{x})+\sum_{i= t+1}^{\lceil \log x/\log 3\rceil}\pi(\sqrt[i]{x}).$$
Now, by the Prime Number Theorem follows
\begin{equation}\label{1}
\frac{\pi(\sqrt[i]{x})}{\sqrt[t]{x}}=\Oh\left(\frac{t}{\sqrt[t(t+1)]{x} \log x}\right), \  \forall i>t\,.
\end{equation}
In fact,
\begin{equation*}
\frac{\pi(\sqrt[i]{x})}{\sqrt[t]{x}}\sim \frac{i\sqrt[i]{x}}{\sqrt[t]{x}\log x}
\end{equation*}
and since $\dfrac{1}{t}-\dfrac{1}{i}>\dfrac{1}{t(t+1)}$ follows the equation (\ref{1}).
Hence we can write
\begin{equation*}\label{Prop1}
\sum_{i=t+1}^{\lceil \log x/\log 3\rceil}\pi(\sqrt[i]{x})=
\Oh\left(\sum_{i=t+1}^{\lceil \log x/\log 3\rceil} \frac{i}{\sqrt[t(t+1)]{x}\log x}\right)=
\Oh\left(\frac{\log x}{\sqrt[t(t+1)]{x}}\right). 
\end{equation*}
Once again, the Prime Number Theorem ensures that $\pi(\sqrt[t]{x})=\oh(\sqrt[t]{x})$, and by the above equation 
we conclude $|\mathcal{U}(x)|=\oh(\sqrt[t]{x})$. Finally
$$\lim_{x\to \infty}\frac{\vert\mathcal{R}_{t}(x)\vert}{\sqrt[t]{x}}=0\,,$$ which concludes the proof. 
\end{proof}


The next lemma was shown also by Klee in \cite{K}, here we give a different proof. Additionally, 
the importance of the next lemma for the present paper also justify to include our proof.

\begin{lemB}\label{lemmaB}
It follows that $\A(2r)\in\{0,2,4\}$, when $2r\equiv2\mod 4$.
If $\A(2r)=2$, then $\phi^{-1}(2r)=\{p^{n}, 2p^{n}\}$, with $p$ an odd prime,
$n>0$. If $\A(2r)=4$, then $2r+1$ is a prime number and 
$\phi^{-1}(2r)=\{2r+1, q^{m}, 4r+2, 2q^{m}\}$ with $q$ a prime number and $m>1$.
\end{lemB}
\begin{proof} 
Let $x$ be a positive integer such that $\phi(x)=2r$ and let $S_x$ be the set of odd prime factors of $x$.
If there are prime numbers $p,q>2$ such that  $p,q\in S_x$, then $2r\equiv 0\mod 4$, which is a contradiction. Thus $x=p^k$ or $x=2p^k$.
Note that $x=2p^k$ is a solution of $\phi(x)=2r$ iff $\phi(x/2)=2r$. Then, we can suppose that $x=p^k,y=q^m$ and $z=t^l$ are solutions of 
          \begin{equation}\label{F-lemma}
          \phi(w)=2r.
          \end{equation}
Assuming that $k,m>1$, from (\ref{F-lemma}) we get
         \begin{equation*}
         p^{k-1}(p-1)=q^{m-1}(q-1)\,.
        \end{equation*}
Since $k,m>1$, $p\vert\, q-1$ and $q\vert\,p-1$, which is a contradiction.
Then we can assume that $k=1$ and therefore $m>1$. The above argument immediately implies that $l=1$. It follows from (\ref{F-lemma})
that $p-1=t-1$. Thus $\A(2r)\in \{0,2,4\}$. Moreover, if $\A(2r)=4$ and $p-1=2r$ implies $p=2r+1$, ie. $2r+1$ is a
prime number. \end{proof}

The following naive remark is addressed to Carmichael's conjecture. It can be taken as another motivation of studying 
the distribution of totients $2$ mod. $4$.

\begin{rem}\label{remark}
Let $m=2^k\cdot r$ be any even positive integer with $r$ odd. If $\A(2r)=4$ then $\A(m)\ge 2$. 
In fact, 
$$2^k \cdot r=2^{k-1} \cdot 2r=\phi(2^k)\phi(x),$$
where $x\in \{p,q^t,2p,2q^t\}$ and $t>1$.
Taking $x=p$ and $x=q^t$ we have that $\phi(x)=m$.
\end{rem}



\section{On the distribution}

In this section we study the distribution of totients $2$ modulo $4$. 
Let us start by taking the following useful sets
$$
\mathcal{T}_{k}=\{2r\,;\, r  \ \mbox{odd and} \ \A(2r)=i\}\ \ \ (k=0,2,4).
$$

\begin{center}
	\begin{table}[H]
		\caption{The number of totients  $2$ mod $4$ $\leq x$ with a fixed multiplicity}
		\begin{tabular}{|c|c|c|c|c|}\hline
			$x$ & $\pi(x)$ & $\vert\mathcal{T}_{2}(x)\vert$ & $\vert\mathcal{T}_{4}(x)\vert$ & $\vert\mathcal{T}_{2}(x)\vert/\pi(x)$ \\ \hline
			$10^3+2$ & $168$ & $87$ & $5$ & $0.517857\dots$ \\ \hline
			$10^4+2$ & $1229$ & $625$ & $8$ & $0.508543\dots$ \\ \hline
			$10^5+2$ & $9592$ & $4831$ & $14$ & $0.503648\dots$ \\ \hline
			$10^6+2$ & $78498$ & $39400$ & $20$  & $0.501923\dots$ \\ \hline 
			$10^7+2$ & $664579$ & $332606$ & $34$ & $0.500476\dots$  \\ \hline
			$10^8+2$ & $5761455$ & $2881495$ & $78$  & $0.500133\dots$ \\ \hline
		\end{tabular}
	\end{table}
\end{center}


\begin{cor}
$\lim_{x\to \infty}\dfrac{\vert\mathcal{T}_{4}(x)\vert}{\sqrt{x}}=0.$
\end{cor}
\begin{proof}
We just have to observe that $\mathcal{T}_{4}(x)\subset \mathcal{R}_{2}(x)$.
\end{proof}

The following two propositions are consequences of lemmas A and B.

\begin{prop}\label{vkl}
	Let $\l>1$ be a positive integer. Let us consider $\mathcal{V}_{k}^{\l}:=\mathcal{T}_t\cap\mathcal{V}_k$ be the set of totients with multiplicity $k$
	such that there is a power of a prime $p^{\l}$ in its inverse image by $\phi$. With this
	$$\lim_{x\rightarrow\infty} \dfrac{\vert\mathcal{V}_{k}^{\l}(x)\vert}{\vert\mathcal{V}_k(x)\vert}=0$$
\end{prop}
\begin{proof}
	From Lemma A we get $\vert\mathcal{V}_k^l(x)\vert=\oh(\sqrt[l]{x})$. By the Prime Number Theorem we have $\sqrt[l]{x}=\oh(\pi(x))$.
	Now, \cite[Thm. 2]{F1} implies that $\pi(x)=\Oh(\vert\mathcal{V}_k(x)\vert)$. Hence $\vert \mathcal{V}_k^l(x)\vert=\oh(\vert \mathcal{V}_k(x)\vert)$.
\end{proof}

The next Proposition is related to cousin primes, i.e., prime numbers $p$ such that $p+4$ is also a prime number.
\begin{prop}\label{pairs}
	Let $\mathcal{C}(x):=\{2r\leq x-4; 2r,2r+4\in\phi(\mathbb{N})\, \mbox{and} \ r \ \mbox{is odd}\}$ be the trunked set of
	the pairs of consecutive totients $2r$ are $2\, \mbox{mod.} 4$, $$\lim_{x\rightarrow\infty}\dfrac{\vert\mathcal{C}(x)\vert}{\sqrt{x}}=0\,.$$
	In particular,
	$$\lim_{x\rightarrow\infty}\dfrac{\vert\mathcal{C}(x)\vert}{\vert\mathcal{V}_k(x)\vert}=0, \ \forall k\ge 2.$$
\end{prop}
\begin{proof}
By considering the sets $\mathcal{C}_1(x)=\{2r; 2r+1,2r+5\in \mathcal{P}(x+1)\}$ and 
$\mathcal{C}_2(x)=\{2r\le x-4; p^t\in \phi^{-1}(2r)\cup \phi^{-1}(2r+4),t>1\}$, we get
$$\mathcal{C}(x)=\mathcal{C}_1(x)\cup \mathcal{C}_2(x).$$
Note that $\mathcal{C}_1(x)$ is the set of $p-1=2r$ such that $p,p+4$ are cousin primes smaller than or equal to $x+1$ and therefore $\vert\mathcal{C}_1(x)\vert=\oh(\sqrt{x})$.
From lemma A follows that $|\mathcal{C}_2(x)|=\oh(\sqrt{x})$, just 
because $\mathcal{C}_2(x)$ is the union of subsets of length equals to $\oh(\sqrt{x})$. Then $\vert\mathcal{C}(x)\vert=\oh(\sqrt{x})$. 
In particular, since $\sqrt{x}=\oh(\pi(x))$ and $\pi(x)=\Oh(\vert\mathcal{V}(x)\vert)$ we have that 
$$\lim_{x\to \infty}\frac{\vert \mathcal{C}(x)\vert}{\vert\mathcal{V}_k(x)\vert}=0.$$
\end{proof}

\begin{rem}
Evidently, we can change $2r+4$ by $2r+d$, where $d$ is even and we get the same result.
\end{rem}
We denote $\mathcal{P}(k,j)=\{p\in \mathcal{P};p\equiv j \, (\mbox{mod.}k)\}$ and $\pi(x;k,j)=|\mathcal{P}(x;k,j)|$. 
It is obvious that $p\in \phi^{-1}(2r)$ with $r\equiv 1(\mbox{mod.}2)$ implies $p\in \mathcal{P}(4,3).$

\begin{thm}\label{thmA}
$$\vert\mathcal{T}_2(x)\vert\sim \frac{\pi(x)}{2}.$$
\end{thm}
\begin{proof}
Let us consider the sets $\mathcal{T}'_2(x)=\{2r\le x; 2r+1\in \mathcal{P} \ \mbox{and}
\ \A(2r)=2\}$ and $\mathcal{T}'_2(x)=\{2r\in \mathcal{T}_2(x); \, r \ \mbox{odd}\}$. It follows from Lemma B that
$$\mathcal{T}_2(x)=\mathcal{T}'_2(x)\cup \mathcal{T}'_2(x).$$      
Therefore,
$$\frac{\vert \mathcal{T}_2(x)\vert}{\pi(x)}=\frac{\vert \mathcal{T}'_2(x)\vert}{\pi(x)}+ \frac{\vert\mathcal{T}'_2(x) \vert}{\pi(x)}.$$
Moreover, from Lemma A we have that $\vert\mathcal{T}'_2(x)\vert =\oh(\sqrt{x})$. Since $\sqrt{x}=\oh(\pi(x))$, 
the Prime Number Theorem implies that $\vert\mathcal{T}'_2(x)\vert =\oh(\pi(x))$. Hence
\begin{equation}
  \liminf_{x\to \infty}\frac{\vert \mathcal{T}_2(x)\vert}{\pi(x)}=\liminf_{x\to \infty}\frac{\vert \mathcal{T}'_2(x)\vert}{\pi(x)} \ \mbox{and} 
  \ \limsup_{x\to \infty}\frac{\vert \mathcal{T}_2(x)\vert}{\pi(x)}=\limsup_{x\to \infty}\frac{\vert \mathcal{T}'_2(x)\vert}{\pi(x)}.
\end{equation}  
\textbf{Claim:}$$\lim_{x\to \infty}\frac{\vert \mathcal{T}'_2(x)\vert}{\pi(x)}=\frac{1}{2}.$$
In fact, since $2r+1$ must be a prime $3(\mbox{mod.} \ 4)$ we have that
$$\{2r+1\in \mathcal{P}(x+1;4,3)\}=\mathcal{T}'_2(x)\cup \{2r+1\in \mathcal{P}(x+1;4,3); \A(2r)=4\}.$$
By the Prime Number Theorem in Arithmetic Progression and above corollary
$$\frac{1}{2}=\lim_{x\to \infty}\frac{\pi(x;4,3)}{\pi(x)}=
\liminf_{x\to \infty}\frac{\mathcal{T}'_2(x)}{\pi(x)}\le \limsup_{x\to \infty}\frac{\mathcal{T}'_2(x)}{\pi(x)}=
\lim_{x\to \infty}\frac{\pi(x;4,3)}{\pi(x)}.$$
This finishes the claim and proves the theorem.
\end{proof}

\begin{rem}
It follows from the above Theorem and the remark \ref{remark} that
$$\vert\{m=2^kr\le x, r \ \mbox{odd and} \ \A(m)\ge 2\}\vert\gg \dfrac{\pi(x)}{2^k}$$ because $\pi(2x)\sim 2\pi(x)$ and $k$ is a fixed integer.
\end{rem}

\begin{cor}\label{corA}
$\vert\mathcal{T}_4(x)\vert=\oh(\vert\mathcal{T}_2(x) \vert).$
\end{cor}
\begin{proof}
Since $\vert\mathcal{T}_2(x)\vert \sim \dfrac{x}{2\log x}$ and $\vert\mathcal{T}_4(x)\vert=\oh(\sqrt{x})$, we get
$$\frac{\vert \mathcal{T}_4(x)\vert}{\vert\mathcal{T}_2(x)\vert}=\Oh\left(\frac{\sqrt{x}\log x}{x}\right)=\Oh\left(\frac{\log x}{\sqrt{x}}\right).$$
Therefore,
$$\lim_{x\to \infty}\frac{\vert\mathcal{T}_4(x)\vert}{\vert\mathcal{T}_2(x)\vert}=0.$$
\end{proof}

\noindent \textbf{Acknowledgment:} The authors warmly thank Paulo Ribenboim for his contagious enthusiasm, constant advices and encouragement.






\bibliographystyle{abbrv}
\bibliography{RefsCL}








\end{document}